\documentclass[12pt]{amsart}
\usepackage{amsfonts}
\usepackage{amsmath}
\usepackage{amssymb}
\usepackage{amsthm,amsfonts,latexsym,epsfig,geometry}
\usepackage{hyperref,times}
\usepackage{amscd}
\usepackage{mathtools}
\usepackage{color}
\newcommand\cplus{\mathbin{\raisebox{-\height}{$+$}}}
\newcommand\contdots{\raisebox{-\height}{$\vphantom{+}\dotsm$}}
\allowdisplaybreaks[1]
\newtheorem{theorem}{Theorem}[section]
\newtheorem{lemma}[theorem]{Lemma}

\newtheorem{corollary}[theorem]{Corollary}
\theoremstyle{definition}
\newtheorem{definition}[theorem]{Definition}

\theoremstyle{remark}

\newtheorem{conj}[theorem]{Conjecture}
\numberwithin{equation}{section}
\usepackage{hyperref}
\hypersetup{
	colorlinks=true,
	linkcolor={blue}, urlcolor={blue},
	citecolor={red}, anchorcolor = {blue}}
\begin{document}
\title[On the Fourier coefficients of Certain Hauptmoduln]{Distribution and divisibility of the Fourier coefficients of certain Hauptmoduln}
\author{Chiranjit Ray}
\address{Mathematics and Data Science Group, Indian Institute of Information Technology Sri City,
Tirupati District - 517 646, Andhra Pradesh, India}
\curraddr{}
\email{chiranjitray.m@gmail.com, chiranjitray.m@iiits.in}

\subjclass[2010]{Primary: 11F03, 11F20, 11F30, 11F33.}
\date{\today}
\keywords{Hauptmoduln; Eta-quotients; Modular forms; q-Series; Rogers-Ramanujan; Continued fraction; Distribution.}
\thanks{} 
\begin{abstract}
Suppose $j_N(\tau)$ and $j_N^{*}(\tau)$ are the  Hauptmoduln of the congruence subgroup $\Gamma_0(N)$ and the Fricke group $\Gamma^{*}_0(N)$, respectively.
In \cite{Kumari2019}, the authors predicted that, like Klein's $j$-function, the Fourier coefficients of  $j_N(\tau)$ and  $j_{N}^{*}(\tau)$ in some arithmetic progression are both even and odd with density $\frac{1}{2}$. In this article, we  can find some arithmetic progression of $n$ where the Fourier coefficients of $j_6(\tau)$ (resp.\ $j_6^{*}(\tau)$ and $j_{10}(\tau)$) are almost always even. Furthermore, using Hecke eigenforms and Rogers-Ramanujan continued fraction, we obtain infinite families of congruences for $j_6(\tau)$, $j_6^{*}(\tau)$, $j_{10}(\tau),$ and $j_{10}^{*}(\tau)$.

\end{abstract}
\maketitle
\section{Introduction and statement of results}
\label{intro}
The elliptic modular $j$-function is defined by
\begin{align*}
	j(\tau):=\frac{E_4^{3}(\tau)}{\Delta(\tau)}= \frac{1}{q}+744+\sum_{n=1}^{\infty}c(n)q^n,
\end{align*}
where $\tau\in \mathbb{H}$, $q=e^{2\pi i \tau}$, $ \Delta(\tau)$  is the modular
discriminant function (also known as Ramanujan's
Delta function) and $E_4(\tau)$  is the normalized Eisenstein series of weight 4. The function $j(\tau)$ has many beautiful properties and plays an important role in number theory. The values of the $j$-function at CM points are known to be algebraic integers. Fourier coefficients of $j(\tau)$ are related to the Monster group, and its CM values generate abelian extensions over imaginary quadratic fields.
\par The study of Fourier coefficients of the elliptic modular functions 
has a long history. Congruence properties of the Fourier coefficients of the modular invariant $j(\tau)$ modulo products of powers of $2, 3, 5, 7, 11$ were first given by Lehmer \cite{Lehmer1942} and Lehner \cite{Lehner1, Lehner2}. Then Newman \cite{Newman1958} derived some congruences modulo $13$.  Many mathematicians like Kolberg \cite{Kolberg1959, Kolberg1962}, Ono and Taguchi \cite{ono_taguchi2005}, and Alfes \cite{Alfes2013} have studied the divisibility and distribution  of the Fourier coefficients of the $j$-function. In recent work, Ono and Ramsey \cite{ono2012} proved that there are infinitely many odd $d$ such that $c(nd^2)$ is even when $n \equiv 7 \pmod 8$ is square-free.

\par Let $\Gamma^{'}$ be a congruence subgroup of $\mathrm{SL}(2,\mathbb{R})$ of level $N$ with genus zero. A Hauptmodul of level $N$ for the congruence subgroup $\Gamma^{'}$ is a generator for the field of modular functions.
Note that the $j$-function is a Hauptmodul for level $1$. We consider the Hauptmoduln $j_N(\tau)$ and $j_N^{*}(\tau)$ with respect to the congruence subgroup $\Gamma_0(N)$ and the Fricke group $\Gamma^{*}_0(N)$, respectively. The group $\Gamma^{*}_0(N)$
is generated by $\Gamma_0(N)$ and the Aktin-Lehner involution $W_{N}=\frac{1}{\sqrt{N}}\begin{pmatrix}
0 & -1 \\
N & 0
\end{pmatrix}$.
It is well-known that for $N=2,3,5,6,7, 10, 13,$ the groups $\Gamma_0(N)$ and $\Gamma^{*}_0(N)$ have genus zero.
The Dedekind eta~function, $\eta(\tau)$, is defined by
\begin{align*}
\eta(\tau):=q^{1/24}\prod_{n=1}^{\infty}(1-q^n),
\end{align*}
where $q=e^{2\pi i\tau}$ and $\tau\in \mathbb{H}$.
 A function $f(\tau)$ is called an eta-quotient if it is of the form
\begin{align*}
f(\tau)=\prod_{\delta\mid N}\eta(\delta \tau)^{r_\delta},
\end{align*}
where $N$ is a positive integer and $r_{\delta}$ is an integer. Matsusaka~\cite{Matsusaka2017} 
observed that the Hauptmoduln $j_N(\tau)$ and $j_N^{*}(\tau)$ can be expressed in terms of the eta-quotient. For $p=2,3,5,7,13,$ we have
\begin{align*}
	j_p(\tau):&=\left(\frac{\eta(\tau)}{\eta(p\tau)}\right)^{\frac{24}{p-1}}+\frac{24}{p-1},  \\
	j_p^{*}(\tau):&=j_p(\tau)+p^{\frac{12}{p-1}}\left(\frac{\eta(p\tau)}{\eta(\tau)}\right)^{\frac{24}{p-1}}, \\
	j_6(\tau):&=\left(\frac{\eta(2\tau)\eta(3\tau)^3}{\eta(\tau)\eta(6\tau)^3}\right)^3-3, \\
	j_6^{*}(\tau):&=\left(\frac{\eta(\tau)\eta(3\tau)}{\eta(2\tau)\eta(6\tau)}\right)^6+6+2^6\left(\frac{\eta(2\tau)\eta(6\tau)}{\eta(\tau)\eta(3\tau)}\right)^6,\\
	j_{10}(\tau):&=\frac{\eta(2\tau)\eta(5\tau)^5}{\eta(\tau)\eta(10\tau)^5}-1,\\
	j_{10}^{*}(\tau):&=\left(\frac{\eta(\tau)\eta(5\tau)}{\eta(2\tau)\eta(10\tau)}\right)^{4}+4+2^4\left(\frac{\eta(2\tau)\eta(10\tau)}{\eta(\tau)\eta(5\tau)}\right)^{4}.
\end{align*}

Suppose $\mathcal{J}_N(n)$ and  $\mathcal{J}^*_N(n)$ denote the $n$-th Fourier coefficients of $j_N(\tau)$ and $j_N^{*}(\tau)$, respectively.  Kumari and Singh \cite{Kumari2019} studied the  parity of the Fourier coefficients of the Hauptmoduln $j_m(\tau)$ and $j_m^{*}(\tau)$ for $m=2,3,4,5,7,13$. For one of these fixed $m$, they showed that at least one Fourier coefficient of $j_m(\tau)$ (resp.\ $j_m^{*}(\tau)$)  exists in suitable residue classes with a fixed parity in suitable intervals. 
For example, they proved that for each positive integer $t$, the interval $\left[t, 4t(t + 1) - 1\right]$ $\left(\text{resp}.\  \left[16t - 1,(4t + 1)^2 - 1\right]
\right)$ contains an integer $n\equiv7\pmod8$ such that $\mathcal{J}_2^{*}(n)$ is odd (resp.\ even). In the same paper, it remarked that the even and odd values of the Fourier coefficients of  $j_N(\tau)$ and  $j_{N}^{*}(\tau)$ in some arithmetic progression is equally distributed. 

In this article, we study the distribution of the Fourier coefficients of $j_6(\tau)$, $j_6^{*}(\tau)$ and $j_{10}(\tau)$ and prove that $\mathcal{J}_6(n)$, $\mathcal{J}_6^{*}(n)$ and $\mathcal{J}_{10}(n)$ are even numbers for almost every non-negative integer $n$ satisfying some arithmetic progression.  Furthermore, we prove some infinite families of congruences and parity results for $\mathcal{J}_6(n)$, $\mathcal{J}_6^{*}(n)$, $\mathcal{J}_{10}(n)$ and $\mathcal{J}_{10}^{*}(n)$. In Section~\ref{Sect:j6 and j6*}, we prove the results for $j_6(\tau)$ and $j_6^{*}(\tau)$. The results for $j_{10}(\tau)$ and $j_{10}^{*}(\tau)$ are obtained in Section~\ref{Sect:j10 and j10*}. Next, we present some of our sample results proven in this article.
	
In the following theorem, we prove that the set of those positive integers $n$ for which $\mathcal{J}^*_{6}(3^m(8n+1))~\equiv~0~\pmod{2}$ has arithmetic density one, where $m$ is either $0$ or $1$. 
\begin{theorem}\label{Thmj6*(24n+3)} Let $n$ be a positive integer and $m\in \{0,1\}.$ Then
\begin{align*}
    \lim\limits_{X\to +\infty}\frac{\# \left\{n\leq X:\mathcal{J}^*_{6}(3^m(8n+1))\equiv 0 \pmod{2} \right\}}{X}&=1.
\end{align*}
\end{theorem}
In other words, for almost every non-negative integer $n$ lying in an arithmetic progression, the integer $\mathcal{J}^*_{6}(3^m(8n+1))$ is even. In fact, there exists a positive constant $\alpha$ such that there are at most $\mathcal{O}\left(\frac{X}{(\log{}X)^{\alpha}}\right)$ many integers $n\leq X$ for which $\mathcal{J}^{*}_6(3^m(8n+1))$ is odd. We obtain similar results for $\mathcal{J}_6(n)$ in Theorem~\ref{Thmj6(24n+3)} and $\mathcal{J}_{10}(n)$ in Theorem~\ref{Thmj10(4n+1)}. 

It is a natural question to ask for an arithmetic progression
$an+b$ such that $\mathcal{J}^{*}_6(an+b)\equiv 0 \pmod2$ hold for each non-negative integer $n$. Next, using the theory of Hecke eigenforms, we find infinite families of arithmetic progressions for which $\mathcal{J}^{*}_6(n)$ is even. Similarly, the results for  $j_6(\tau)$ and $j_{10}(\tau)$ are stated in Theorem~\ref{Thmj6(24n+3).2} and  Theorem~\ref{Thmj10.2}, respectively.

\begin{theorem}\label{Thmj*6(24n+3).2} 
Let $k, n$ be non-negative integers. For each $i$ with $1\leq i \leq k+1$, consider the prime numbers $p_i$ such that  $p_i  \equiv 3, 5, 7 \pmod {8}$. Then, for any integer $j$ not divisible by ${p_{k+1}}$, we have
 	\begin{align*} 
 	\mathcal{J}^*_{6}\Big(3^mp_1^2p_2^2\dots p_{k}^2p_{k+1}\left(8p_{k+1}n+8j+p_{k+1}\right)\Big) &\equiv 0 \pmod 2,
 	\end{align*}
 where $m$ is either $0$ or $1$.	
 \end{theorem}
 \noindent One can get different infinite family of congruences from Theorem \ref{Thmj*6(24n+3).2}. Let $p$ be a prime such that $p\equiv 3,5,7\pmod{8}$. Suppose  $p_1= p_2= \ldots=p_{k+1}=p.$ Then from Theorem \ref{Thmj*6(24n+3).2}, 
we obtain the following infinite family of congruences for $\mathcal{J}^*_{6}(n)$:
\begin{align*} 
 \mathcal{J}^*_{6}\Big(3^mp^{2k+1}\left(8pn+8j+p\right)\Big) &\equiv 0 \pmod 2,
\end{align*}
where  $j \not\equiv 0 \pmod p$ and $m\in\{0,1\}$.\\  In particular, for all $n\geq 0,$ $m=0$ and $j\not\equiv 0\pmod{3}$, \begin{align*}
\mathcal{J}^*_{6}\left(72n + 24j + 9\right) \equiv 0 \pmod{2}.
\end{align*}
 
 Our next result shows that the Fourier coefficients of $j^*_6(\tau)$ have the same parity on two different numbers $N_1$ and $N_2$. We get similar results for $j_6(\tau)$ in Theorem~\ref{Thmj6(24n+3).3} and $j_{10}(\tau)$ in Theorem~\ref{Thmj10.3}.
\begin{theorem}\label{Thmj*6(24n+3).3} 	Let $k$ be a positive integer and $i \in \{3, 5, 7\}$. Suppose $p$ is a prime number such that $p \equiv i \pmod {8}$. 
Let $ \delta $ be a non-negative integer such that $p$ divides $8\delta  + i$ and $m \in \{0,1\}$, then $\mathcal{J}^*_{6}\left(N_1 \right)$ and  $\mathcal{J}^*_{6}\left(N_2\right)$ have the same parity, where $N_1=3^mp\left(8p^{k}n+ 8\delta+i\right)$ and  $N_2 = 3^m\left(8p^{k-1}n+ \frac{8\delta +i-p}{p}\right)+3.$
 \end{theorem}
As a special case of the above theorem, we obtain the following result.
\begin{corollary}\label{coroj*6(24n+3)}	Let $k$ be a positive integer and $p$ be a prime number such that $p  \equiv 3, 5, 7 \pmod {8}$. Then
$\mathcal{J}^*_{6}\left( 3^{m}(8n+1\right)$ and $\mathcal{J}^*_{6}\left(3^{m}p^{2k}\left( 8n+1\right)\right)$ have the same parity, where $m\in \{0,1\}.$
\end{corollary}

Next, we give a list of Fourier coefficients of  $j^*_{6}(\tau)$ and $j_6(\tau)$, which have the same parity under certain conditions. It is interesting to see these patterns in the Fourier coefficients of these Hauptmoduln.

\begin{theorem}\label{Thmj6=j*6} For any positive integer $n$, we have
\begin{itemize}
     \item[(a).] $(2n)$-th Fourier coefficients of  $j^*_{6}(\tau)$ and $j_6(\tau)$ are always even.
     \item[(b).] $(4n+1)$-th Fourier coefficients of $j_6(\tau)$ are always even.
     \item[(c).] $(24n+11)$-th and $(24n+19)$-th Fourier coefficients of  $j^*_{6}(\tau)$ and $j_6(\tau)$ are always even.
    \item[(d).] $(4n-1)$-th Fourier coefficients of  $j^*_{6}(\tau)$ and $j_6(\tau)$ have the same parity.
    \item[(e).] $(8n+1)$-th Fourier coefficients of  $j^*_{6}(\tau)$ and $(24n+3)$-th Fourier coefficients of $j_6(\tau)$ have the same parity.
\end{itemize}
\end{theorem}
On the same spirit, using  a formula associated with Rogers-Ramanujan continued fraction, we prove the following identities for the Fourier coefficients of $j_{10}(\tau)$ and $j^*_{10}(\tau)$.

\begin{theorem}\label{j10(z)=j*10(z).1.1}
Let $i\in\{7,23\}$ and $n$ be a positive number. Then $(40n+i)$-th 
Fourier coefficients of both $j_{10}(\tau)$ and $j^*_{10}(\tau)$ are even numbers.
\end{theorem}  

We use {\it Mathematica} \cite{mathematica} for necessary computations.
\section{Preliminaries}
In this section, we recall some definitions, and facts relating to the arithmetic of classical modular forms and for  more details, one can consult \cite{koblitz1993, ono2004}.  Let $\mathbb{H}$ denotes the  upper-half plane.

The complex vector space of modular forms of weight $\ell$ (a positive integer) with respect to a congruence subgroup $\Gamma$ will be denoted by $M_{\ell}(\Gamma)$.
\begin{definition}\cite[Definition 1.15]{ono2004}
	Let $\chi$ be a Dirichlet character modulo $N$ (a positive integer). Then a modular form $f\in M_{\ell}(\Gamma_1(N))$ has Nebentypus character $\chi$ if
	$$f\left( \frac{az+b}{cz+d}\right)=\chi(d)(cz+d)^{\ell}f(\tau)$$ for all $z\in \mathbb{H}$ and all $\begin{pmatrix}
	a  &  b \\
	c  &  d      
	\end{pmatrix} \in \Gamma_0(N)$. The space of such modular forms is denoted by $M_{\ell}(\Gamma_0(N), \chi)$.  Here $\Gamma_0(N)$ will as usual be the principal congruence subgroup of level $N$.
\end{definition}
 
\par 
We now recall two theorems from \cite[p.\ 18]{ono2004} that help us  check the modularity of eta-quotients  which show up  in our study.
\begin{theorem}\cite[Theorem 1.64]{ono2004}\label{thm_ono1} If $f(\tau)=\displaystyle\prod_{\delta\mid N}\eta(\delta \tau)^{r_\delta}$ 
	is an eta-quotient such that 
\begin{align*}	
\ell&=\displaystyle\frac{1}{2}\sum_{\delta\mid N}r_{\delta}\in \mathbb{Z},\\
\sum_{\delta\mid N} \delta r_{\delta}&\equiv 0 \pmod{24}\\
 \intertext{and}
\sum_{\delta\mid N} \frac{N}{\delta}r_{\delta}&\equiv 0 \pmod{24}.
\end{align*}
Then 
$$
f\left( \frac{az+b}{cz+d}\right)=\chi(d)(cz+d)^{\ell}f(\tau)
$$
	for every  $\begin{pmatrix}
	a  &  b \\
	c  &  d      
	\end{pmatrix} \in \Gamma_0(N)$. Here 
	$$\chi(d):=\left(\frac{(-1)^{\ell} \prod_{\delta\mid N}\delta^{r_{\delta}}}{d}\right).$$
	\end{theorem}
 Suppose that $f$ is an eta-quotient satisfying the conditions of Theorem \ref{thm_ono1}.  Over and above, if $f$ is also holomorphic at all of the cusps of $\Gamma_0(N)$, then $f\in M_{\ell}(\Gamma_0(N), \chi)$. To check the holomorphicity at cusps of $f(\tau)$ it suffices to check that the orders at the cusps are non-negative. The necessary criterion for determining orders of an eta-quotient at cusps is given in the next result.
\begin{theorem}\cite[Theorem 1.65]{ono2004}\label{thm_ono1.1}
Let $c, d,$ and $N$ are positive integers with $d\mid N$ and $\gcd(c, d)=1$. If $f(\tau)$ is an eta-quotient satisfying the conditions of Theorem~\ref{thm_ono1} for $N$, then the order of vanishing of $f(\tau)$ at the cusp $\frac{c}{d}$ 
	is $$\frac{N}{24}\sum_{\delta\mid N}\frac{\gcd(d,\delta)^2r_{\delta}}{\gcd(d,\frac{N}{d})d\delta}.$$
\end{theorem}
A Hecke operator acts as a natural linear transformation on the spaces of modular forms. Let us recall the definition of Hecke operators for integer weight modular forms.  
\begin{definition}
	Let $m$ be a positive integer and $f(\tau) = \displaystyle\sum_{n=0}^{\infty} a(n)q^n \in M_{\ell}(\Gamma_0(N),\chi)$. The Hecke operator $T_m$ acts on $f(\tau)$ by 
	\begin{align*}
	f(\tau)|T_m := \sum_{n=0}^{\infty} \left(\sum_{d\mid \gcd(n,m)}\chi(d)d^{\ell-1}a\left(\frac{nm}{d^2}\right)\right)q^n.
	\end{align*}
	In particular, if $m=p$ is a  prime, then
	\begin{align}\label{hecke1}
	f(\tau)|T_p := \sum_{n=0}^{\infty} \left(a(pn)+\chi(p)p^{\ell-1}a\left(\frac{n}{p}\right)\right)q^n.
	\end{align}
\end{definition}
\begin{definition}\label{hecke2}
	A modular form $f(\tau) \in M_{\ell}(\Gamma_0(N),\chi)$ is called a Hecke eigenform if for every $m\geq2$ there exists a complex number $\lambda(m)$ for which 
	\begin{align}\label{hecke3}
	f(\tau)|T_m = \lambda(m)f(\tau).
	\end{align}
\end{definition}
Now we state a consequence of the binomial theorem which we deliberately used in many places. For any positive integer $k$, we have 
$$(q^{2k};q^{2k})_{\infty}\equiv(q^{k};q^{k})_{\infty}^2\pmod 2,$$
where the $q$-shifted
factorial $\displaystyle (a; q)_{\infty}= \prod_{j=0}^{\infty}(1-aq^j),~$ $|q|<1$.

\section{ Arithmetic properties  of the Fourier coefficients of  \texorpdfstring{$j_{6}(\tau)$}{j6} and  \texorpdfstring{$j^{*}_{6}(\tau)$}{j*6}}\label{Sect:j6 and j6*}
In this section, we consider the hauptmoduln \texorpdfstring{$j_{6}(\tau)$}{j6} and  \texorpdfstring{$j^{*}_{6}(\tau)$}{j*6}. First, we study the distribution property and divisibility of the Fourier coefficients of \texorpdfstring{$j_{6}(\tau)$}{j6}. Then we obtain the results for the Fourier coefficients of \texorpdfstring{$j^{*}_{6}(\tau)$}{j*6} as stated in Section~\ref{intro}.

Recall that the hauptmodul \texorpdfstring{$j_{6}(\tau)$}{j6} has the following form
\begin{align*}
j_6(\tau):&=\left(\frac{\eta(2\tau)\eta(3\tau)^3}{\eta(\tau)\eta(6\tau)^3}\right)^3-3\\
&=\frac{1}{q}\cdot\left(\frac{(q^2;q^2)_{\infty}(q^3;q^3)_{\infty}^3}{(q;q)_{\infty}(q^{6};q^{6})_{\infty}^3}\right)^3-3\\ 
&=\frac{1}{q}+6 q + 4 q^2 - 3 q^3 - 12 q^4 - 8 q^5 + 12 q^6 + 30 q^7+\cdots\\
&=\frac{1}{q}+\sum_{n=1}^{\infty} \mathcal{J}_{6}(n)q^n.	
\end{align*}
Let us define \begin{align*}
\sum_{n=0}^{\infty}\mathcal{F}_{6}(n)q^n&= \left(\frac{(q^2;q^2)_{\infty}(q^3;q^3)_{\infty}^3}{(q;q)_{\infty}(q^{6};q^{6})_{\infty}^3}\right)^3\\
&=1 + 3 q + 6 q^2 + 4 q^3 - 3 q^4 - 12 q^5 - 8 q^6 + 12 q^7 + 30 q^8+\cdots.
\end{align*}
Then for any positive integer $n$,
\begin{align}\label{j6n=f6(n+1)}
  \mathcal{J}_{6}(n)=\mathcal{F}_{6}(n+1).
\end{align}
Next we prove the following lemma for $\mathcal{F}_{6}(n)$ which would be useful later.
\begin{lemma}\label{f624n+4=tau}For any positive integer $n$, we have 
\begin{align*}
\sum_{n=0}^{\infty}\mathcal{F}_{6}(24n+4)q^{8n+1}&\equiv \eta(8\tau)\eta(16\tau) \pmod 2,
\end{align*} 
where the eta~product $\eta(8\tau)\eta(16\tau)\in S_1\left(\Gamma_0(128), (\frac{-2}{\bullet})\right)$.
\end{lemma}
\begin{proof} using the binomial theorem, we have
\begin{align}\label{f6n}
\notag \sum_{n=0}^{\infty}\mathcal{F}_{6}(n)q^n&= \left(\frac{(q^2;q^2)_{\infty}(q^3;q^3)_{\infty}^3}{(q;q)_{\infty}(q^{6};q^{6})_{\infty}^3}\right)^3\\
\notag &\equiv \frac{(q;q)_{\infty}^3(q^3;q^3)_{\infty}^9}{(q^6;q^6)_{\infty}^9}\pmod 2\\
\notag &\equiv \frac{(q;q)_{\infty}^3}{(q^3;q^3)_{\infty}^9}\pmod 2\\
&\equiv \frac{(q;q)_{\infty}(q^2;q^2)_{\infty}}{(q^3;q^3)_{\infty}^3(q^6;q^6)_{\infty}^3}\pmod 2.
\end{align}
In \cite[Lemma 2.5]{Yao2013JNT}, we have  following 2-dissections formula:

\begin{align}
\label{f1/f3^3}
\frac{(q;q)_{\infty}}{(q^{3};q^{3})_{\infty}^3}&= \frac{(q^2;q^2)_{\infty}(q^4;q^4)_{\infty}^2(q^{12};q^{12})_{\infty}^2}{(q^{6};q^{6})_{\infty}^7} - q \frac{(q^2;q^2)_{\infty}^3(q^{12};q^{12})_{\infty}^6}{(q^4;q^4)_{\infty}^2(q^6;q^6)_{\infty}^9}.
\end{align}
From \eqref{f6n} and \eqref{f1/f3^3}, we obtain
\begin{align}\label{f6n.2}
\notag \sum_{n=0}^{\infty}\mathcal{F}_{6}(n)q^n&\equiv \frac{(q^2;q^2)_{\infty}}{(q^6;q^6)_{\infty}^3}\left(\frac{(q^2;q^2)_{\infty}(q^4;q^4)_{\infty}^2(q^{12};q^{12})_{\infty}^2}{(q^{6};q^{6})_{\infty}^7} - q \frac{(q^2;q^2)_{\infty}^3(q^{12};q^{12})_{\infty}^6}{(q^4;q^4)_{\infty}^2(q^6;q^6)_{\infty}^9}\right)\\
&=\frac{(q^2;q^2)_{\infty}^2(q^4;q^4)_{\infty}^2(q^{12};q^{12})_{\infty}^2}{(q^{6};q^{6})_{\infty}^{10}} - q \frac{(q^2;q^2)_{\infty}^4(q^{12};q^{12})_{\infty}^6}{(q^4;q^4)_{\infty}^2(q^6;q^6)_{\infty}^{12}}.
\end{align}
Extracting the terms containing the even power of $q$ from \eqref{f6n.2} and then using the binomial theorem, we have
\begin{align}\label{f62n.1}
\sum_{n=0}^{\infty}\mathcal{F}_{6}(2n)q^n&\equiv\frac{(q;q)_{\infty}^2(q^2;q^2)_{\infty}^2(q^{6};q^{6})_{\infty}^2}{(q^{3};q^{3})_{\infty}^{10}}\equiv \frac{(q^2;q^2)_{\infty}^3}{(q^6;q^6)_{\infty}^3} \pmod 2.
\end{align}
Note that the series expansion of $\dfrac{(q^2;q^2)_{\infty}^3}{(q^6;q^6)_{\infty}^3}$ have only even powers of $q$. Thus from \eqref{f62n.1}, we obtain
\begin{align} \label{f64n}
\sum_{n=0}^{\infty}\mathcal{F}_{6}(4n)q^n&\equiv \frac{(q;q)_{\infty}(q^2;q^2)_{\infty}}{(q^3;q^3)_{\infty}^3} \pmod 2.
\end{align}
First using \eqref{f1/f3^3} and \eqref{f64n}, and then extracting the terms containing  $q^{2n+1}$, we have
\begin{align}
\label{f68n+4}\sum_{n=0}^{\infty}\mathcal{F}_{6}(8n+4)q^n&\equiv (q^3;q^3)_{\infty}^3 \pmod 2.
\end{align}
Now collecting the terms of $q^{3n}$ from \eqref{f68n+4} and using the binomial theorem, we obtain
\begin{align}
\label{f624n+4}\sum_{n=0}^{\infty}\mathcal{F}_{6}(24n+4)q^n&\equiv (q;q)_{\infty}(q^2;q^2)_{\infty} \pmod 2.
\end{align}
 
 Let us consider an eta-quotient $\eta(8\tau)\eta(16\tau)$. 
 Using Theorem \ref{thm_ono1} and Theorem \ref{thm_ono1.1}, we  see that eta-quotient $\eta(8\tau)\eta(16\tau)$ is a cusp form of weight 1, and level $128$ with a Nebentypus character, i.e., $$\eta(8\tau)\eta(16\tau)\in S_1\left(\Gamma_0(128), \chi_{-2}\right),$$ where $\chi_{-2}$ is defined by $\chi_{-2}(\bullet)=(\frac{-2}{\bullet}).$
 From \eqref{f624n+4}, we have
\begin{align*}
\sum_{n=0}^{\infty}\mathcal{F}_{6}(24n+4)q^{8n+1}&\equiv \eta(8\tau)\eta(16\tau) \pmod 2.
\end{align*} 
This concludes the lemma.
\end{proof}
Next we prove the following density result for $\mathcal{J}_{6}(n)$.
\begin{theorem}\label{Thmj6(24n+3)} For any positive integer $n$,
\begin{align*}
    \lim\limits_{X\to +\infty}\frac{\# \left\{n\leq X:\mathcal{J}_{6}(24n+3)\equiv 0 \pmod{2} \right\}}{X}=1.
\end{align*}
\end{theorem}
The following theorem of Serre is useful to prove our result.
	\begin{theorem}\cite[Theorem 2.65]{ono2004}\label{Serre}
		Let $A$ denote the subset of integer weight modular forms in $M_{k}(\Gamma_0(N),\chi)$
 whose Fourier coefficients are in $\mathcal{O}_K$, the ring of algebraic integers in a number field $K$. Suppose $ \mathcal{M} \subset \mathcal{O}_K$ is an ideal. If $f(\tau) \in A$ has Fourier expansion
	$$
f(\tau)=\sum_{n=0}^{\infty}a(n)q^n
$$
then for every $\mathcal{M}$, there is a constant $\alpha>0$  such that
$$ \# \left\{n\leq X: a(n)\not\equiv 0 \pmod{\mathcal{M}} \right\}= \mathcal{O}\left(\frac{X}{(\log{}X)^{\alpha}}\right).$$
\end{theorem}
\begin{proof}[Proof of Theorem \ref{Thmj6(24n+3)}] Suppose $\eta(8\tau)\eta(16\tau)$ has a Fourier series expansion $\sum_{n=0}^{\infty} \mathcal{A}(n)q^n.$ From Lemma~\ref{f624n+4=tau}, we have
\begin{align*}
\notag \sum_{n=0}^{\infty}\mathcal{F}_{6}(24n+4)q^{8n+1}&\equiv \eta(8\tau)\eta(16\tau) \pmod 2\\
&=\sum_{n=0}^{\infty} \mathcal{A}(n)q^n.
\end{align*} 
Therefore,
\begin{align}\label{pmodf624n+4.23}
\mathcal{F}_{6}(24n+4)\equiv \mathcal{A}(8n+1) \pmod{2}.
\end{align}
Since $\eta(8\tau)\eta(16\tau)$ is an integer weight cusp form in $S_1\left(\Gamma_0(128), (\frac{-2}{\bullet})\right)$ with integer Fourier coefficients, by Theorem \ref{Serre} we have a constant $\alpha>0$ such that
\begin{align}\label{iiits.12}
 \# \left\{n\leq X:\mathcal{A}(n)\not\equiv 0 \pmod{2} \right\}= \mathcal{O}\left(\frac{X}{(\log{}X)^{\alpha}}\right).   
\end{align}
Since $\mathcal{A}(n) = 0$ if $n\not\equiv 1\pmod{8},$ we have
$$ \left\{n\leq 8X+1:\mathcal{A}(n)\not\equiv 0 \pmod{2} \right\} = \left\{n\leq 8X+1:\mathcal{A}(n)\not\equiv 0 \pmod{2},~ n\equiv 1 \pmod{8}  \right\}.$$
If we consider a  map
$$f: \left\{n\leq 8X+1:\mathcal{A}(n)\not\equiv 0 \pmod{2},~ n\equiv 1 \pmod{8}  \right\}\rightarrow \left\{n\leq X:\mathcal{A}(8n+1)\not\equiv 0 \pmod{2} \right\},$$ defined by $f(n)=\frac{n-1}{8},$ then $f$ is bijective. Therefore,
\begin{align}
\label{iiits.11}
   \# \left\{n\leq 8X+1:\mathcal{A}(n)\not\equiv 0 \pmod{2} \right\} = \# \left\{n\leq X:\mathcal{A}(8n+1)\not\equiv 0 \pmod{2} \right\}. 
\end{align}
Using \eqref{iiits.12} and \eqref{iiits.11}, we have
\begin{align*}
 \# \left\{n\leq X:\mathcal{A}(8n+1)\not\equiv 0 \pmod{2} \right\}= \mathcal{O}\left(\frac{X}{(\log{}X)^{\alpha}}\right).   
\end{align*}
From \eqref{pmodf624n+4.23}, we obtain $$\lim\limits_{X\to +\infty}\frac{	\# \left\{n\leq X:\mathcal{F}_{6}(24n+4)\equiv 0 \pmod{2} \right\}}{X}=1.$$
Finally using \eqref{j6n=f6(n+1)} we complete the proof.
\end{proof}
Next, we obtain the following infinite families of congruences for the Fourier coefficients of $j_6(\tau)$ modulo~$2$ using the theory of Hecke eigenforms. 
\begin{theorem}\label{Thmj6(24n+3).2} 
Let $k, n$ be non-negative integers. For each $i$ with $1\leq i \leq k+1$, consider the prime numbers $p_i$ such that  $p_i\equiv 3,5,7 \pmod {8}$. Then, for any integer $j$ not divisible by ${p_{k+1}}$, we have
 	\begin{align*} 
 	\mathcal{J}_{6}\Big(3p_1^2p_2^2\dots p_{k}^2p_{k+1}\left(8p_{k+1}n+8j+p_{k+1}\right)\Big) \equiv 0 \pmod 2.
 	\end{align*}
 \end{theorem}
\noindent In particular, if we consider  $k=0$ and $j\not\equiv 0\pmod{5}$, then for all $n\geq 0,$ we have the following congruence:
\begin{align*}
\mathcal{J}_{6}\left(600n + 120j + 75\right) \equiv 0 \pmod{2}.
\end{align*} 
To prove Theorem \ref{Thmj6(24n+3).2}, we need the following lemma where we find some arithmetic properties of the Fourier coefficients of $\eta(8\tau)\eta(16\tau)$.
\begin{lemma}\label{lemma2.1.2}
	Suppose $\eta(8\tau)\eta(16\tau)$ has a Fourier series expansion $\displaystyle\sum_{n=1}^{\infty} \mathcal{A}(n)q^n$ and $p$ be a prime number such that $p \equiv 3,5,7 \pmod {8}$.  Then, 
	\begin{align*}	
	\mathcal{A}(p^2n + pr) &= 0  \;\;\; \text{if} \;\;\;  0<r<p,\\
	\intertext{and}
\mathcal{A}(p^2n) + \left(\frac{-2}{p}\right) \mathcal{A}(n)&= 0.
	\end{align*}	
\end{lemma}
\begin{proof}
	We have $\eta(8\tau)\eta(16\tau) = q - q^9 - 2 q^{17} + q^{25} + 2 q^{41} + q^{49} - 2 q^{73} +  \dots=\displaystyle\sum_{n=1}^{\infty} \mathcal{A}(n)q^n$. It is easy to observe that  $\mathcal{A}(n) = 0$ if $n\not\equiv 1\pmod{8}$.  From \cite{martin} we know that  $\eta(8\tau)\eta(16\tau)$ is a Hecke eigenform. Using \eqref{hecke1} and \eqref{hecke3}, we get
		\begin{align*}
	\eta(8\tau)\eta(16\tau)|T_p = \sum_{n=1}^{\infty} \left(\mathcal{A}(pn) + \left(\frac{-2}{p}\right) \mathcal{A}\left(\frac{n}{p}\right) \right)q^n = \lambda(p) \sum_{n=1}^{\infty} \mathcal{A}(n)q^n.
	\end{align*}
	Equating the coefficients on the both sides, we have the following
	\begin{align}\label{1.1.1}
	\mathcal{A}(pn) + \left(\frac{-2}{p}\right) \mathcal{A}\left(\frac{n}{p}\right) = \lambda(p)\mathcal{A}(n).
	\end{align}
Since we consider the prime numbers $p \equiv 3,5,7 \pmod{8}$, putting $n=1$ in \eqref{1.1.1},  we obtain $\mathcal{A}(p) = 0 = \lambda(p)$. Therefore, from \eqref{1.1.1}, we have 
	\begin{align}\label{new-1.1}
	\mathcal{A}(pn) + \left(\frac{-2}{p}\right)\mathcal{A}\left(\frac{n}{p}\right)= 0
	\end{align}
 for all prime $p \equiv 3,5,7 \pmod {8}.$  
 Now, either $p\nmid n$ or $p\mid n$. We conclude the lemma from \eqref{new-1.1} by replacing 
 
 	\begin{align*}	
	n & \rightarrow pn+r \;\;\; \text{if} \;\;\;   0<r<p\\
	\intertext{or}
	n & \rightarrow  pn. 
	\end{align*}

	\end{proof}
\begin{proof}[Proof of Theorem \ref{Thmj6(24n+3).2}] 
If $p\nmid n$, substituting $n$ by $8n-pr+1$ in Lemma \ref{lemma2.1.2}  and using the congruence relation  \eqref{pmodf624n+4.23}, we have
	\begin{align*}
	\mathcal{F}_{6}\left(24p^2n + 3(p^2+pr-p^3r)+1\right) \equiv 0\pmod {2}.
	\end{align*}	
Here we consider the  prime  $p \equiv 3, 5, 7\pmod{8}$ and $\gcd \left(\frac{p^2-1}{8}, p\right) = 1$. Hence if $r$ runs over a residue system excluding the multiple of $p$, so does $\frac{1-p^2}{8}r$.
For $j\not\equiv0 \pmod {p}$, we can rewrite the above equation as
	\begin{align}\label{1.6}
	\mathcal{F}_{6}\left(24p^2n + 3p^2 + 24pj+1\right) \equiv 0\pmod {2}.
	\end{align}
Similarly, substituting $n$ by $8n+1$ in Lemma \ref{lemma2.1.2}  and using the congruence relation \eqref{pmodf624n+4.23}, we obtain
	\begin{align*}
	\mathcal{F}_{6}\left(24p^2n + 3p^2+1\right)\equiv \mathcal{F}_{6}\left(24n + 4\right)\pmod {2}.
	\end{align*}
For $1\leq i\leq k$, consider any primes $p_i\equiv3,5,7\pmod8.$ Therefore,  
\begin{align*}
	&\mathcal{F}_{6}\big(24p_1^2p_2^2\dots p_{k}^2n + 3p_1^2p_2^2\dots p_{k}^2+1\big)\\
	&= \mathcal{F}_{6}\left(24p_1^2\left(p_2^2\dots p_{k}^2n + \frac{p_2^2\dots p_{k}^2-1}{8}\right)+ 3p_1^2+1\right)\\
	&\equiv \mathcal{F}_{6}\left(24p_2^2\dots p_{k}^2n + 3p_2^2\dots p_{k}^2+1\right)\pmod 2.
	\end{align*}
By using the above recursive relation for $(k-1)$~times, we have
\begin{align}\label{1.7}
	&\mathcal{F}_{6}\big(24p_1^2p_2^2\dots p_{k}^2n + 3p_1^2p_2^2\dots p_{k}^2+1\big)\equiv \mathcal{F}_{6}\left(24n+4\right)\pmod 2.
	\end{align}
Let us consider a prime $p_{k+1}\equiv 3, 5, 7\pmod 8$ and $j\not\equiv 0\pmod{p_{k+1}}$. Then, \eqref{1.6} and \eqref{1.7} yield
\begin{align*}
	\mathcal{F}_{6}\left(24p_1^2p_2^2\dots p_{k}^2p_{k+1}^2n+24p_1^2p_2^2\dots p_{k}^2p_{k+1}j+3p_1^2p_2^2\dots p_{k}^2p_{k+1}^2+1\right)\equiv 0\pmod 2.
	\end{align*}
Hence we readily obtain Theorem \ref{Thmj6(24n+3).2} from the above congruences and \eqref{j6n=f6(n+1)}.
\end{proof}

\begin{theorem}\label{Thmj6(24n+3).3} 	Let $k$ be a positive integer and $i \in \{3, 5, 7\}$. Suppose $p$ is a prime number such that $p \equiv i \pmod {8}$. 
Let $ \delta $ be a non-negative integer such that $p$ divides $8\delta  + i$, then $\mathcal{J}_{6}\left(N_1 \right)$ and  $\mathcal{J}_{6}\left(N_2\right)$ have the same parity, where $N_1=3p\left(8p^{k}n+ 8\delta+i\right)$ and $N_2=24\left(p^{k-1}n+ \frac{8\delta +i-p}{8p}\right)+3.$
 \end{theorem}
\begin{proof}
Let $p$ be a prime number with $p\equiv 3,5,7 \pmod 8$. Then for all positive integer $n$, we have the following from  \eqref{new-1.1}. 
		$$\mathcal{A}\left(pn\right) \equiv \mathcal{A}\left(\frac{n}{p}\right) \pmod 2.$$	
Let $i$ be a fixed number such that  $ i \in \{3,5,7\}$. Suppose  $\delta$ is a positive number such that  the prime $p\equiv i \pmod {8}$ divides $8\delta  + i$. Then, substituting $n$ by $8p^kn+8\delta+i$ in the above relation and using \eqref{pmodf624n+4.23}, modulo $2$, we obtain. 
\begin{align}\label{thm2.2.4.1}
\mathcal{F}_{6}\left(24\left( p^{k+1}n+ p\delta +\frac{pi-1}{8}\right)+4\right)\equiv \mathcal{F}_{6}\left(24\left(p^{k-1}n+ \frac{8\delta +i-p}{8p}\right)+4\right).
\end{align}	
In general the above congruences occurs because $\frac{pi-1}{8}$ and $\frac{8\delta +i-p}{8p}$ are integers. The result directly follows from \eqref{j6n=f6(n+1)} and \eqref{thm2.2.4.1}.
\end{proof}
\begin{corollary}\label{coroj6(24n+3)}	Let $k$ be a positive integer and $p$ be a prime number such that $p  \equiv 3, 5, 7 \pmod {8}$. Then
$\mathcal{J}_{6}\left( 24n+3\right)$ and $\mathcal{J}_{6}\left(3p^{2k}\left( 8n+1\right)\right)$    have the same parity.
\end{corollary}
\begin{proof}Let $p$ be a prime  such that $p\equiv i\pmod {8}$ where $ i \in \{3, 5, 7\}$. Depending on the residue of $p$ modulo 8, we consider a non-negative integer $\delta$ with $8\delta+i=p^{2k-1}.$ First, substituting $n$ by $p^{k-1}n$ in \eqref{thm2.2.4.1}, we get
\begin{align*}
\mathcal{F}_{6}\left(24\left( p^{2k}n+ p\delta +\frac{pi-1}{8}\right)+4\right)\equiv \mathcal{F}_{6}\left(24\left(p^{2(k-1)}n+ \frac{8\delta +i-p}{8p}\right)+4\right).
\end{align*}	
Replacing $8\delta+i$ by $p^{2k-1}$, we have
\begin{align*}
\mathcal{F}_{6}\left(3p^{2k}\left( 8n+1\right)+1\right)\equiv \mathcal{F}_{6}\left(3p^{2(k-1)}\left( 8n+1\right)+1\right)\pmod 2.
\end{align*}	
By using the above recursive relation for $(k-1)$ times, we obtain \begin{align*}
\mathcal{F}_{6}\left(3p^{2k}\left( 8n+1\right)+1\right)\equiv \mathcal{F}_{6}\left(3\left( 8n+1\right)+1\right)\pmod 2.
\end{align*}
Corollary \ref{coroj6(24n+3)} directly follows from the above equation and \eqref{j6n=f6(n+1)}. 
\end{proof}
Next we consider the Hauptmodul $j_6^{*}(\tau)$ given as follows and get the result for its $n$-th Fourier coefficient $\mathcal{J}^*_{6}(n)$.
\begin{align*}
j_6^{*}(\tau):&=\left(\frac{\eta(\tau)\eta(3\tau)}{\eta(2\tau)\eta(6\tau)}\right)^6+6+2^6\left(\frac{\eta(2\tau)\eta(6\tau)}{\eta(\tau)\eta(3\tau)}\right)^6\\ 
&=\frac{1}{q}+79 q + 352 q^2 + 1431 q^3 + 4160 q^4 + 13015 q^5 + 31968 q^6+\cdots\\
&=\frac{1}{q} + \sum_{n=1}^{\infty} \mathcal{J}^*_{6}(n)q^n.	
\end{align*}
Let us define \begin{align}
\label{f*6n.2.3.4}\sum_{n=0}^{\infty}\mathcal{F}^*_{6}(n)q^n&= \left(\frac{(q;q)_{\infty}(q^3;q^3)_{\infty}}{(q^2;q^2)_{\infty}(q^{6};q^{6})_{\infty}}\right)^6\\
\notag &=1 - 6 q + 15 q^2 - 32 q^3 + 87 q^4 - 192 q^5 + 343 q^6 - 672 q^7+\cdots.
\end{align}
Therefore, for any positive integer $n$ we have 
\begin{align}\label{j*6n=f*6(n+1)}
  \mathcal{J}^*_{6}(n)\equiv \mathcal{F}^*_{6}(n+1) \pmod 2.
\end{align}
To prove Theorems \ref{Thmj6*(24n+3)}--\ref{Thmj6=j*6}, we need the following Lemma.
\begin{lemma}\label{Thmj6*(24n+3).lemma}
For any positive integer $n$, we have 
\begin{align*}
  \sum_{n=0}^{\infty}\mathcal{F}^*_{6}(4n)q^n&\equiv \frac{(q;q)_{\infty}^3}{(q^3;q^3)_{\infty}^3}\pmod 2\\
 \intertext{and}
   \sum_{n=0}^{\infty}\mathcal{F}^*_{6}(8n+2)q^n &\equiv(q;q)_{\infty}^3 \pmod 2.
\end{align*}
\end{lemma}
\begin{proof}
Using the binomial theorem in \eqref{f*6n.2.3.4}, we have
\begin{align}\label{f*6n}
\sum_{n=0}^{\infty}\mathcal{F}^*_{6}(n)q^n&= \left(\frac{(q;q)_{\infty}(q^3;q^3)_{\infty}}{(q^2;q^2)_{\infty}(q^{6};q^{6})_{\infty}}\right)^6\equiv \left(\frac{1}{(q^2;q^2)_{\infty}(q^6;q^6)_{\infty}}\right)^3\pmod 2.
\end{align}
Xia and Yao~\cite[p.380]{Xia2013}  derived the following 2-dissection formula 
\begin{align}\label{1/f1f3}
\frac{1}{(q;q)_{\infty}(q^{3};q^{3})_{\infty}}= & \frac{(q^8;q^8)_{\infty}^2(q^{12};q^{12})_{\infty}^5}{(q^2;q^2)_{\infty}^2(q^4;q^4)_{\infty}(q^6;q^6)_{\infty}^4(q^{24};q^{24})_{\infty}^2}\\
\notag&+q\frac{(q^4;q^4)_{\infty}^5(q^{24};q^{24})_{\infty}^2}{(q^2;q^2)_{\infty}^4(q^6;q^6)_{\infty}^2(q^8;q^8)_{\infty}^2(q^{12};q^{12})_{\infty}}.
\end{align}
Extracting the terms from \eqref{f*6n} containing the even power of $q$ and then using the 2-dissection formula \eqref{1/f1f3} we have
\begin{align}\label{f*62n}
\notag \sum_{n=0}^{\infty}\mathcal{F}^*_{6}(2n)q^n&\equiv \frac{1}{(q^2;q^2)_{\infty}(q^6;q^6)_{\infty}}\left(\frac{(q^8;q^8)_{\infty}^2(q^{12};q^{12})_{\infty}^5}{(q^2;q^2)_{\infty}^2(q^4;q^4)_{\infty}(q^6;q^6)_{\infty}^4(q^{24};q^{24})_{\infty}^2}\right.\\
\notag&~~~\left. + q\frac{(q^4;q^4)_{\infty}^5(q^{24};q^{24})_{\infty}^2}{(q^2;q^2)_{\infty}^4(q^6;q^6)_{\infty}^2(q^8;q^8)_{\infty}^2(q^{12};q^{12})_{\infty}}\right) \pmod 2\\
& = \frac{(q^8;q^8)_{\infty}^2(q^{12};q^{12})_{\infty}^5}{(q^2;q^2)_{\infty}^3(q^4;q^4)_{\infty}(q^6;q^6)_{\infty}^5(q^{24};q^{24})_{\infty}^2} + q\frac{(q^4;q^4)_{\infty}^5(q^{24};q^{24})_{\infty}^2}{(q^2;q^2)_{\infty}^5(q^6;q^6)_{\infty}^3(q^8;q^8)_{\infty}^2(q^{12};q^{12})_{\infty}}.
\end{align}
Now extracting the terms containing the even power of $q$ from \eqref{f*62n} and using the binomial theorem, we obtain
\begin{align*}
\sum_{n=0}^{\infty}\mathcal{F}^*_{6}(4n)q^n&\equiv\frac{(q^4;q^4)_{\infty}^2(q^{6};q^{6})_{\infty}^5}{(q;q)_{\infty}^3(q^2;q^2)_{\infty}(q^3;q^3)_{\infty}^5(q^{12};q^{12})_{\infty}^2} \equiv\frac{(q;q)_{\infty}^3}{(q^3;q^3)_{\infty}^3} \pmod 2.
\end{align*}
Which complete the proof of the first part of Lemma \ref{Thmj6*(24n+3).lemma}.\\

\noindent To prove the second part of Lemma \ref{Thmj6*(24n+3).lemma}, we use the  following 2-dissections formula (see \cite[Lemma 2.5]{Yao2013JNT}):
\begin{align}\label{f3^3/f1}
\frac{(q^{3};q^{3})_{\infty}^3}{(q;q)_{\infty}}&= \frac{(q^4;q^4)_{\infty}^3(q^{6};q^{6})_{\infty}^2}{(q^2;q^2)_{\infty}^2(q^{12};q^{12})_{\infty}} + q \frac{(q^{12};q^{12})_{\infty}^3}{(q^4;q^4)_{\infty}}.
\end{align}
Collecting the terms containing the odd power of $q$ from \eqref{f*62n} and using the 2-dissections formula \eqref{f3^3/f1}, we obtain
\begin{align*}
\notag\sum_{n=0}^{\infty}\mathcal{F}^*_{6}(4n+2)q^n&\equiv \frac{(q^2;q^2)_{\infty}^5(q^{12};q^{12})_{\infty}^2}{(q;q)_{\infty}^5(q^3;q^3)_{\infty}^3(q^4;q^4)_{\infty}^2(q^{6};q^{6})_{\infty}}\\
\notag &\equiv \frac{1}{(q^2;q^2)_{\infty}} \frac{(q^3;q^3)_{\infty}^3}{(q;q)_{\infty}}\pmod 2\\
&= \frac{1}{(q^2;q^2)_{\infty}}\left(\frac{(q^4;q^4)_{\infty}^3(q^{6};q^{6})_{\infty}^2}{(q^2;q^2)_{\infty}^2(q^{12};q^{12})_{\infty}} + q \frac{(q^{12};q^{12})_{\infty}^3}{(q^4;q^4)_{\infty}}\right).
\end{align*}
Extracting the terms containing $q^{2n}$ from the above equation, we get
\begin{align}
\label{1121}
\sum_{n=0}^{\infty}\mathcal{F}^*_{6}(8n+2)q^n\equiv\frac{(q^2;q^2)_{\infty}^3(q^{3};q^{3})_{\infty}^2}{(q;q)_{\infty}^3(q^{6};q^{6})_{\infty}} \pmod 2.
\end{align}
Next, using the binomial theorem in \eqref{1121}, we readily obtain the last part of Lemma \ref{Thmj6*(24n+3).lemma}.
\end{proof}

Observe that we establish two natural congruence relations between  the generating function of $\mathcal{F}_6(n)$ and $\mathcal{F}^*_6(n)$ modulo~$2$. From \eqref{f64n}, \eqref{f624n+4} and Lemma~\ref{Thmj6*(24n+3).lemma}, we have
\begin{align}
  \label{f64n=f*64n}  \sum_{n=0}^{\infty}\mathcal{F}^*_{6}(4n)q^n&\equiv \sum_{n=0}^{\infty}\mathcal{F}_{6}(4n)q^n\equiv \frac{(q;q)_{\infty}^3}{(q^3;q^3)_{\infty}^3}\pmod 2\\
 \intertext{and}
  \label{f*6(8n+2)=f6(24n+4)} \sum_{n=0}^{\infty}\mathcal{F}^*_{6}(8n+2)q^n &\equiv \sum_{n=0}^{\infty}\mathcal{F}_{6}(24n+4)q^n\equiv(q;q)_{\infty}^3 \pmod 2.
\end{align}
\begin{proof}[Proof of Theorems \ref{Thmj6*(24n+3)}-\ref{Thmj6=j*6}]
Using the relationships \eqref{f64n=f*64n} and \eqref{f*6(8n+2)=f6(24n+4)}, we readily obtain Theorem~\ref{Thmj6*(24n+3)}, Theorem \ref{Thmj*6(24n+3).2},  Theorem~\ref{Thmj*6(24n+3).3}, and  Corollary \ref{coroj*6(24n+3)} from Theorem \ref{Thmj6(24n+3)},  Theorem \ref{Thmj6(24n+3).2}, Theorem \ref{Thmj6(24n+3).3}, and Corollary \ref{coroj6(24n+3)}, respectively. 

Now it remains to prove Theorem~\ref{Thmj6=j*6}.
Extracting the terms containing the odd power of $q$ from \eqref{f6n.2},  \eqref{f62n.1}  and \eqref{f*6n}, we get
\begin{align*}
\mathcal{F}_{6}(2n+1)&\equiv 0 \pmod 2,\\
\mathcal{F}_{6}(4n+2)&\equiv 0 \pmod 2\\
\intertext{and}
\mathcal{F}^*_{6}(2n+1)&\equiv 0 \pmod 2,
\end{align*}
respectively. We readily obtain Theorem \ref{Thmj6=j*6} (a)-(b), by using  \eqref{j6n=f6(n+1)},  \eqref{j*6n=f*6(n+1)} and above congruences.\\

For $i\in \{1,2\}$, extracting the terms containing $q^{3n+i}$  from \eqref{f68n+4}, we have
\begin{align}\label{f64n=f*64n.2.2}
\mathcal{F}_{6}(8(3n+i)+4)&\equiv 0 \pmod 2.
\end{align}
Now it follows from \eqref{f64n=f*64n} that $\mathcal{F}_{6}(4n)$ and $\mathcal{F}^*_{6}(4n)$ have the same parity. Then using \eqref{j6n=f6(n+1)}, \eqref{j*6n=f*6(n+1)} and \eqref{f64n=f*64n.2.2}, we get Theorem~\ref{Thmj6=j*6} (c). Lastly, the parts (d)-(e)  of Theorem~\ref{Thmj6=j*6} directly follows from  \eqref{f64n=f*64n} and \eqref{f*6(8n+2)=f6(24n+4)}.
 \end{proof}
\section{ Arithmetic properties  of the Fourier coefficients of  \texorpdfstring{$j_{10}(\tau)$}{j10} and  \texorpdfstring{$j^{*}_{10}(\tau)$}{j*10}} \label{Sect:j10 and j10*}
In this section, we consider the hauptmoduln \texorpdfstring{$j_{10}(\tau)$}{j10} and  \texorpdfstring{$j^{*}_{10}(\tau)$}{j*10}. We study the distribution and divisibility of the Fourier coefficients of those hauptmoduln. 
\begin{theorem}\label{Thmj10(4n+1)} For any positive integer $n$, we have 
\begin{align*}
    \lim\limits_{X\to +\infty}\frac{\# \left\{n\leq X:\mathcal{J}_{10}(4n+1)\equiv 0 \pmod{2} \right\}}{X}=1.
\end{align*}
\end{theorem}
\begin{proof}
We have
\begin{align*}
j_{10}(\tau):&=\frac{\eta(2\tau)\eta(5\tau)^5}{\eta(\tau)\eta(10\tau)^5}-1\\
&=\frac{1}{q}+q+2q^2+2q^3-2q^4-q^5-4q^7+\cdots\\
&=\frac{1}{q} + \sum_{n=1}^{\infty}\mathcal{J}_{10}(n)q^n.	
\end{align*}
Let us consider $$\sum_{n=0}^{\infty}\mathcal{F}_{10}(n)q^n= \frac{(q^2;q^2)_{\infty}(q^5;q^5)_{\infty}^5}{(q;q)_{\infty}(q^{10};q^{10})_{\infty}^5}=1+q+q^2+2q^3+2q^4-2q^5-q^6-4q^8+\cdots.$$ 
Therefore, for any positive integer $n$ we have 
\begin{align}\label{j10=f10}
\mathcal{J}_{10}(n)=\mathcal{F}_{10}(n+1).
\end{align}
Using the binomial theorem we have 
\begin{align}\label{F10(n)}
\sum_{n=0}^{\infty}\mathcal{F}_{10}(n)q^n= \frac{(q^2;q^2)_{\infty}(q^5;q^5)_{\infty}^5}{(q;q)_{\infty}(q^{10};q^{10})_{\infty}^5}\equiv \frac{(q;q)_{\infty}}{(q^{5};q^{5})_{\infty}(q^{20};q^{20})_{\infty}}\pmod2.
\end{align}
Xia and Yao~\cite[p. 391]{Xia2013} proved the following  $2$-dissection formula
\begin{align*}
\frac{(q^5;q^5)_{\infty}}{(q;q)_{\infty}}= \frac{(q^8;q^8)_{\infty}(q^{20};q^{20})_{\infty}^2}{(q^2;q^2)_{\infty}^2(q^{40};q^{40})_{\infty}}+q\frac{(q^4;q^4)_{\infty}^3(q^{10};q^{10})_{\infty}(q^{40};q^{40})_{\infty}}{(q^2;q^2)_{\infty}^3(q^{8};q^{8})_{\infty}(q^{20};q^{20})_{\infty}}.
\end{align*}
Replacing $q$ by $-q$ in the above formula, we obtain
\begin{align}\label{f5/f1}
	\frac{(q;q)_{\infty}}{(q^5;q^5)_{\infty}} = \frac{(q^2;q^2)_{\infty}(q^8;q^8)_{\infty}(q^{20};q^{20})_{\infty}^3}{(q^4;q^4)_{\infty}(q^{10};q^{10})_{\infty}^3(q^{40};q^{40})_{\infty}}-q\frac{(q^4;q^4)_{\infty}^2(q^{40};q^{40})_{\infty}}{(q^{8};q^{8})_{\infty}(q^{10};q^{10})_{\infty}^2}.
\end{align}
It follows from \eqref{F10(n)} and \eqref{f5/f1} that 
\begin{align}\label{F10(n).1}
	\sum_{n=0}^{\infty}\mathcal{F}_{10}(n)q^n\equiv \frac{1}{(q^{20};q^{20})_{\infty}}\left(\frac{(q^2;q^2)_{\infty}(q^8;q^8)_{\infty}(q^{20};q^{20})_{\infty}^3}{(q^4;q^4)_{\infty}(q^{10};q^{10})_{\infty}^3(q^{40};q^{40})_{\infty}}-q\frac{(q^4;q^4)_{\infty}^2(q^{40};q^{40})_{\infty}}{(q^{8};q^{8})_{\infty}(q^{10};q^{10})_{\infty}^2}\right).
\end{align}
Now, we extract the terms containing $q^{2n}$ from \eqref{F10(n).1}, and  use binomial theorem and the $2$-dissection formula~\eqref{f5/f1}, to obtain the following congruences relation.
\begin{align}\label{F10(2n)}
\notag	\sum_{n=0}^{\infty}\mathcal{F}_{10}(2n)q^n &\equiv \frac{(q;q)_{\infty}(q^4;q^4)_{\infty}(q^{10};q^{10})_{\infty}^2}{(q^2;q^2)_{\infty}(q^{5};q^{5})_{\infty}^3(q^{20};q^{20})_{\infty}}\\
\notag	&\equiv \frac{(q;q)_{\infty}(q^2;q^2)_{\infty}}{(q^5;q^5)_{\infty}(q^{10};q^{10})_{\infty}}\\
	&=\frac{(q^2;q^2)_{\infty}}{(q^{10};q^{10})_{\infty}}\left(\frac{(q^2;q^2)_{\infty}(q^8;q^8)_{\infty}(q^{20};q^{20})_{\infty}^3}{(q^4;q^4)_{\infty}(q^{10};q^{10})_{\infty}^3(q^{40};q^{40})_{\infty}}-q\frac{(q^4;q^4)_{\infty}^2(q^{40};q^{40})_{\infty}}{(q^{8};q^{8})_{\infty}(q^{10};q^{10})_{\infty}^2}\right).
\end{align}
Extracting the terms containing the odd power of $q$ from \eqref{F10(2n)}, we have 
\begin{align}
\label{f10(4n+2)} \notag	\sum_{n=0}^{\infty}\mathcal{F}_{10}(4n+2)q^n &\equiv \frac{(q;q)_{\infty}(q^2;q^2)_{\infty}^2(q^{20};q^{20})_{\infty}}{(q^4;q^4)_{\infty}(q^{5};q^{5})_{\infty}^3}\\
	 &\equiv (q;q)_{\infty} (q^5;q^5)_{\infty} \pmod{2}.
\end{align}
Let us consider an eta-quotient $\eta(4\tau)\eta(20\tau)$. Using Theorem \ref{thm_ono1} and Theorem \ref{thm_ono1.1}, we  see that eta-quotient $\eta(4\tau)\eta(20\tau)$ is a cusp form of weight 1, and level $80$ with  certain Nebentypus character, i.e., $$\eta(4\tau)\eta(20\tau)\in S_1\left(\Gamma_0(80), \chi_{-20}\right),$$ where $\chi_{-20}$ is defined by $\chi_{-20}(\bullet)=(\frac{-20}{\bullet})$. From \eqref{f10(4n+2)}, we have
\begin{align}\label{f10(4n+2)=tau}
\sum_{n=0}^{\infty}\mathcal{F}_{10}(4n+2)q^{4n+1}&\equiv \eta(4\tau)\eta(20\tau) \pmod 2.
\end{align} 
Using Theorem \ref{Serre}, we can find a constant $\alpha>0$ such that
\begin{align*}\# \left\{n\leq X:\mathcal{F}_{10}(4n+2)\not\equiv 0 \pmod{2} \right\}= \mathcal{O}\left(\frac{X}{(\log{}X)^{\alpha}}\right).
\end{align*}
Hence  

\begin{align}
\label{f4n+2.new.1}
   \lim\limits_{X\to +\infty}\frac{	\# \left\{n\leq X:\mathcal{F}_{10}(4n+2)\equiv 0 \pmod{2} \right\}}{X}=1. 
\end{align}
Finally, we complete the proof of the theorem by \eqref{j10=f10} and \eqref{f4n+2.new.1}.
\end{proof}
Next, we prove the following infinite families of congruences for the Fourier coefficients of $j_{10}(\tau)$.

\begin{theorem}\label{Thmj10.2} 
Let $k, n$ be non-negative integers. For each $i$ with $1\leq i \leq k+1$, consider the prime numbers $p_i$ such that $p_i  \equiv 3\pmod {4}$. Then, for any integer $j$ not divisible by ${p_{k+1}}$, we have
 	\begin{align*} 
 	\mathcal{J}_{10}\Big(p_1^2p_2^2\dots p_{k}^2p_{k+1}\big(4p_{k+1}n+4j+p_{k+1}\big)\Big) \equiv 0 \pmod 2.
 	\end{align*}
 \end{theorem}
 
 \noindent For example, if we consider  $k=0$ and  $j\not\equiv 0\pmod{3}$ in  Theorem~\ref{Thmj10.2} then for all $n\geq 0,$ we obtain the following simple congruence:
\begin{align*}
\mathcal{J}_{10}\left(36n + 12j + 9\right) \equiv 0 \pmod{2}.
\end{align*} 
 \begin{proof}[Proof of Theorem \ref{Thmj10.2}]
 Let us consider the Fourier series expansion of  $\eta(4\tau)\eta(20\tau)$ as $\displaystyle\sum_{n=0}^{\infty} \mathcal{B}(n)q^n$, i.e.,
\begin{align}\label{eta=B}
 \eta(4\tau)\eta(20\tau)=
\sum_{n=0}^{\infty} \mathcal{B}(n)q^n.  
\end{align} 
It is easy to observe that  $\mathcal{B}(n) = 0$ if $n\not\equiv 1\pmod{4}$.
From \eqref{f10(4n+2)=tau} and \eqref{eta=B}, we have
\begin{align}\label{pmodf10B}
\mathcal{F}_{10}(4n+2)\equiv \mathcal{B}(4n+1) \pmod{2}.
\end{align}
In \cite{martin}, Martin proved that  $\eta(4\tau)\eta(20\tau)$ is a Hecke eigenform. Using \eqref{hecke1} and \eqref{hecke3}, we get
		\begin{align*}
	\eta(4\tau)\eta(20\tau)|T_p = \sum_{n=1}^{\infty} \left(\mathcal{B}(pn) + \left(\frac{-20}{p}\right) \mathcal{B}\left(\frac{n}{p}\right) \right)q^n = \lambda(p) \sum_{n=1}^{\infty} \mathcal{B}(n)q^n.
	\end{align*}
 Now using a similar argument as in Lemma \ref{lemma2.1.2}, for a prime $p \equiv3\pmod {4}$,  we have
 
\begin{align}\label{new-1.1.2.1.5}
	\mathcal{B}(pn) + \left(\frac{-20}{p}\right) \mathcal{B}\left(\frac{n}{p}\right) = 0.
\end{align}
Now, either $p\nmid n$ or $p\mid n$. Therefore, from \eqref{new-1.1.2.1.5}, we have
	\begin{align*}	
	\mathcal{B}(p^2n + pr) &= 0 \;\;\;   \text{if} \;\;\;  0<r<p,\\
	 \intertext{and}
\mathcal{B}(p^2n) + \left(\frac{-20}{p}\right) \mathcal{B}(n)&= 0.
	\end{align*}
 Next we follow the proof of Theorem \ref{Thmj6(24n+3).2} and use the identity \eqref{j10=f10} and \eqref{pmodf10B} to obtain the result.
 \end{proof}

\begin{theorem}\label{Thmj10.3} Let $k$ be a positive integer and $p$ be a prime number such that $p \equiv 3 \pmod {4}$. Let $ \delta $ be a non-negative integer such that $p$ divides $4\delta +3$, then	 $\mathcal{J}_{10}\left(N_1 \right)$ and  $\mathcal{J}_{10}\left(N_2\right)$ have the same parity, where $N_1=p\left(4p^{k}n+ 4\delta+3\right)$ and $N_2=\left(4p^{k-1}n+ \frac{4\delta +3}{p}\right).$
 \end{theorem}
 \begin{proof}
For any prime $p\equiv 3 \pmod 4$. Let us choose a non-negative integer $ \delta $  such that $p$ divides $4\delta +3$. Substituting $n$ by $4p^kn + 4\delta+ 3$ in \eqref{new-1.1.2.1.5}, we obtain	
\begin{align}\label{new-1.1.2.1}
	\mathcal{B}(4p^{k+1}n + 4p\delta+ 3p)\equiv  \mathcal{B}\left(4p^{k-1}n+ \frac{4\delta +3}{p}\right) \pmod 2.
	\end{align}	
Since $\dfrac{3p-1}{4}$ and $\dfrac{4\delta+3-p}{4p}$ are integers, the result directly follows from \eqref{j10=f10} and \eqref{pmodf10B}. 
 \end{proof}
 
As a special case of the above theorem, we obtain the following result.
\begin{corollary}\label{coroj10}	Let $k$ be a positive integer and $p$ be a prime number such that $p  \equiv 3 \pmod {4}$. Then
$\mathcal{J}_{10}\left( 4n+1\right)$ and $\mathcal{J}_{10}\left(p^{2k}\left( 4n+1\right)\right)$    have the same parity.
\end{corollary}
\begin{proof}
We replace $k$ by $2k-1$ in \eqref{new-1.1.2.1} then substitute $4\delta +3$ by $p^{2k-1}$,  we have
\begin{align*}
\mathcal{B}\big(p^{2k}(4n+1)\big)&\equiv  \mathcal{B}\big(p^{2(k-1)}(4n+1)\big) \pmod 2\\
&~~~\vdots\\
&\equiv  \mathcal{B}(4n+1)\pmod 2.
\end{align*}
The result directly follows \eqref{j10=f10} and \eqref{pmodf10B}.
\end{proof}
 
Next we give a certain values of $n$ where the Fourier coefficients of $j_{10}(\tau)$ and $j_{10}^{*}(\tau)$ are even.
 \begin{theorem}\label{Thmj10=j*10.1} For any positive integer $n$, we have
\begin{itemize}
     \item[(a).] $(2n)$-th, $(8n+3)$-th Fourier coefficients of   $j_{10}(\tau)$ are always even.
     \item[(b).] For $i\in\{0,1,2\}$, the $(4n+i)$-th Fourier coefficients of $j^{*}_{10}(\tau)$ are always even,
\end{itemize}
\end{theorem}
 \begin{proof}
 Extracting the terms containing the even power of $q$ from \eqref{F10(2n)}, we have modulo $2$
\begin{align}\label{f104n}
\notag	\sum_{n=0}^{\infty}\mathcal{F}_{10}(4n)q^n &\equiv \frac{(q;q)_{\infty}^2(q^4;q^4)_{\infty}(q^{10};q^{10})_{\infty}^3}{(q^2;q^2)_{\infty}(q^{5};q^{5})_{\infty}^4(q^{20};q^{20})_{\infty}}\\
	&\equiv \frac{(q^4;q^4)_{\infty}}{(q^{10};q^{10})_{\infty}}.
\end{align}
Extracting the terms containing the odd power of $q$ from \eqref{F10(n).1} and \eqref{f104n}, we have 
\begin{align*}
\mathcal{F}_{10}(2n+1) &\equiv 0 \pmod{2}\\
\intertext{and}
\mathcal{F}_{10}(8n+4) &\equiv 0 \pmod{2},
\end{align*}
respectively.
The first part of Theorem \ref{Thmj10=j*10.1} is readily obtained by using \eqref{j10=f10} and the above two congruences.\\

Now  we consider the Hauptmodul $j_{10}^{*}(\tau)$ given as follows and get the result for its $n$-th Fourier coefficient $\mathcal{J}^*_{10}(n)$. 

\begin{align}\label{j*10}
\notag j_{10}^{*}(\tau)&:=\left(\frac{\eta(\tau)\eta(5\tau)}{\eta(2\tau)\eta(10\tau)}\right)^{4}+4+2^4\left(\frac{\eta(2\tau)\eta(10\tau)}{\eta(\tau)\eta(5\tau)}\right)^{4}\\
\notag &= \frac{1}{q} + 22 q + 56 q^2 + 177 q^3 + 352 q^4 + 870 q^5  +\cdots\\
  &=\frac{1}{q} +\sum_{n=0}^{\infty}\mathcal{J}^{*}_{10}(n)q^n.
\end{align}
Let us consider 
\begin{align}\label{F*10(n)}
 \sum_{n=0}^{\infty}\mathcal{F}^{*}_{10}(n)q^n&= \left(\frac{(q;q)_{\infty}(q^5;q^5)_{\infty}}{(q^2;q^2)_{\infty}(q^{10};q^{10})_{\infty}}\right)^4\\
\notag &=1 - 4 q + 6 q^2 - 8 q^3 + 17 q^4 - 32 q^5 + 54 q^6 +\cdots.
\end{align}
Therefore, from \eqref{j*10} and \eqref{F*10(n)} for any non-negative integer $n$, we have 
\begin{align}\label{j*10=f*10}
    \mathcal{J}^{*}_{10}(n)\equiv\mathcal{F}^{*}_{10}(n
    +1)\pmod{2}.
\end{align}
Using the binomial theorem on \eqref{F*10(n)}, we get
\begin{align} \label{f*10(n).1}
\sum_{n=0}^{\infty}\mathcal{F}^{*}_{10}(n)q^n\equiv \frac{1}{(q^4;q^4)_{\infty}(q^{20};q^{20})_{\infty}}\pmod{2}.
\end{align}
Now extracting the terms containing $q^{4n+i}$ from \eqref{f*10(n).1}, we obtain
\begin{align*}
\notag \sum_{n=0}^{\infty}\mathcal{F}^{*}_{10}(4n+i)q^n &\equiv 0 \pmod{2}, 
\end{align*}
where $i=1,2,3$. Now using \eqref{j*10=f*10} and above inequality, we obtain the second part of Theorem~\ref{Thmj10=j*10.1}.
\end{proof}
The Rogers-Ramanujan continued fraction usually define by
\begin{align*}
  \mathcal{R}(q):= \dfrac{q^{1/5}}{1}\cplus \frac{q}{1}\cplus\frac{q^2}{1}\cplus  \frac{q^3}{1}\cplus\contdots=q^{1/5}\dfrac{(q;q^5)_\infty(q^4;q^5)_\infty}
  {(q^2;q^5)_\infty(q^3;q^5)_\infty}, ~~~|q|<1. 
  \end{align*}
We use the following 5-dissections formulas of $ (q;q)_\infty$  associate with Rogers-Ramanujan continued fraction to prove Theorem \ref{j10(z)=j*10(z).1.1}.
\begin{lemma}\cite[p. 165]{Berndt_spirit} \label{R(q)}If $R(q)=\dfrac{q^{1/5}}{\mathcal{R}(q)}$, then we have
\begin{align*}
  (q;q)_\infty&=(q^{25};q^{25})_\infty\bigg(R(q^5)-q-\frac{q^2}{R(q^5)}\bigg).
\end{align*}
\end{lemma}
\begin{proof}[Proof of Theorem \ref{j10(z)=j*10(z).1.1}]
 Extracting the terms containing the even power of $q$ from \eqref{f104n}, we have
\begin{align}
\label{f10(8n)}	
\sum_{n=0}^{\infty}\mathcal{F}_{10}(8n)q^n &\equiv 
	\frac{(q^2;q^2)_{\infty}}{(q^{5};q^{5})_{\infty}} \pmod{2}.
\end{align} 
Similarly, collecting the terms containing $q^{4n}$ from \eqref{f*10(n).1}, we obtain 
\begin{align}\label{F*10(4n).22}
\sum_{n=0}^{\infty}\mathcal{F}^{*}_{10}(4n)q^n &\equiv \frac{1}{(q;q)_{\infty}(q^5;q^5)_{\infty}}\pmod{2}.
\end{align}
From Tang and Xia~\cite[Lemma 3.2]{Xia2020}, we have an explicit formula for $(q;q)_{\infty}(q^5;q^5)_{\infty}$,
\begin{align*}
  (q;q)_{\infty}(q^5;q^5)_{\infty}=\frac{(q^4;q^4)_{\infty}^2(q^{10};q^{10})_{\infty}^5}{(q^2;q^2)_{\infty}(q^5;q^5)_{\infty}^2(q^{20};q^{20})_{\infty}^2}-q\frac{(q^2;q^2)_{\infty}^5(q^{20};q^{20})_{\infty}^2}{(q;q)_{\infty}^2(q^4;q^4)^2_{\infty}(q^{10};q^{10})_{\infty}}.
\end{align*}
Now replacing $q$ by $-q$ in the above formula, we obtain
\begin{align*}
  \frac{1}{(q;q)_{\infty}(q^5;q^5)_{\infty}}\frac{(q^2;q^2)_{\infty}^3(q^{10};q^{10})_{\infty}^3}{(q^4;q^4)_{\infty}(q^{20};q^{20})_{\infty}}
  &=\frac{(q^4;q^4)_{\infty}^2(q^5;q^5)_{\infty}^2(q^{10};q^{10})_{\infty}^5(q^{20};q^{20})_{\infty}^2}{(q^2;q^2)_{\infty}(q^{10};q^{10})_{\infty}^6(q^{20};q^{20})_{\infty}^2}\\
  &+q\frac{(q;q)_{\infty}^2(q^2;q^2)_{\infty}^5(q^4;q^4)_{\infty}^2(q^{20};q^{20})_{\infty}^2}{(q^2;q^2)_{\infty}^6(q^4;q^4)^2_{\infty}(q^{10};q^{10})_{\infty}}. 
\end{align*}
Therefore,
\begin{align}\label{1/f1f5}
  \frac{1}{(q;q)_{\infty}(q^5;q^5)_{\infty}} = \frac{(q^4;q^4)_{\infty}^3(q^5;q^5)_{\infty}^2(q^{20};q^{20})_{\infty}}{(q^2;q^2)_{\infty}^4(q^{10};q^{10})_{\infty}^4}
 &+q\frac{(q;q)_{\infty}^2(q^4;q^4)_{\infty}(q^{20};q^{20})_{\infty}^3}{(q^2;q^2)_{\infty}^4(q^{10};q^{10})_{\infty}^4}.
\end{align}
Next, putting the formula \eqref{1/f1f5} in \eqref{F*10(4n).22} and using the binomial theorem, we have 
\begin{align}\label{f*10(4n).2.1}
 \sum_{n=0}^{\infty}\mathcal{F}^{*}_{10}(4n)q^n
\equiv\frac{(q^4;q^4)_{\infty}}{(q^{10};q^{10})_{\infty}}
 +q\frac{(q^{20};q^{20})_{\infty}}{(q^2;q^2)_{\infty}} \pmod{2}
\end{align}
Extracting the terms containing $q^{2n}$ from \eqref{f*10(4n).2.1}, we obtain
\begin{align}
\label{f*8n}\sum_{n=0}^{\infty}\mathcal{F}^{*}_{10}(8n)q^n &\equiv  
 \frac{(q^2;q^2)_{\infty}}{(q^{5};q^{5})_{\infty}} \pmod{2}.
\end{align}
Therefore, \eqref{f10(8n)} and \eqref{f*8n} implies the following congruence relations
\begin{align} \label{equal.1.1}
\sum_{n=0}^{\infty}\mathcal{F}_{10}(8n)q^n &\equiv  \sum_{n=0}^{\infty}\mathcal{F}^{*}_{10}(8n)q^n\equiv 
 \frac{(q^2;q^2)_{\infty}}{(q^{5};q^{5})_{\infty}} \pmod{2}.
\end{align}
Now we focus on studying the arithmetic properties of $\dfrac{(q^2;q^2)_{\infty}}{(q^{5};q^{5})_{\infty}}$ using a $5$-dissections formulas of $(q;q)_\infty$  associated with Rogers-Ramanujan continued fraction. From Lemma $\ref{R(q)}$, we have 

\begin{align}\label{best1.2}
\frac{(q^2;q^2)_{\infty}}{(q^{5};q^{5})_{\infty}}=  \frac{(q^{50};q^{50})_\infty}{(q^{5};q^{5})_\infty}\bigg(R(q^{10})-q^2-\frac{q^4}{R(q^{10})}\bigg).
\end{align}
Since there are no terms of $q$ involving the power of $5n+1$ and $5n+3$ in \eqref{best1.2},  for $i\in \{8, 24\}$, we get
\begin{align*}
 \mathcal{F}_{10}(40n+i)&\equiv 0  \pmod 2,\\
 \intertext{and}
  \mathcal{F}^{*}_{10}(40n+i)&\equiv 0  \pmod 2.
\end{align*}
Hence, Theorem \ref{j10(z)=j*10(z).1.1}~ directly follows from \eqref{j10=f10}, \eqref{j*10=f*10}, and above two congruences.
\end{proof}
\begin{corollary}
For any positive integer $n$, the $(40n+15)$-th Fourier coefficients of $j_{10}(\tau)$, and the $(8n+3)$-th and $(40n+15)$-th Fourier coefficients of $j^*_{10}(\tau)$ have the same parity.
\end{corollary}
\begin{proof}Extracting the odd power of $q$ from 
\eqref{f*10(4n).2.1} and the terms containing $q^{5n+2}$ from  \eqref{best1.2}, and using the fact \eqref{equal.1.1}, we obtain
\begin{align*}
\mathcal{F}^{*}_{10}(8n+4) \equiv  \mathcal{F}^{*}_{10}(40n+16)\equiv  \mathcal{F}_{10}(40n+16) \pmod{2}.
\end{align*}
The result readily follows from \eqref{j10=f10} and \eqref{j*10=f*10}.
\end{proof}

\subsection*{Remark} 
In  Theorem \ref{Thmj10=j*10.1} and Theorem \ref{j10(z)=j*10(z).1.1}, we showed that  
$\mathcal{J}^{*}_{10}(4n+j)$ and  $\mathcal{J}^{*}_{10}(40n+i)$  are even when $j\in\{0,1,2\}$ and  $i\in\{7,23\}$, respectively. Although the entire distribution for the Fourier coefficients of $j^{*}_{10}(\tau)$ is yet to study, we propose the following conjecture depending on our numerical calculation.
\begin{conj} Let $A$ and $B$ be two positive integers such that $A\equiv 0\pmod8$, $B\equiv 3\pmod8$.
Then there is an arithmetic progression $An+B$ such that
\begin{align*}
    \lim\limits_{X\to +\infty}\frac{\# \left\{n\leq X:\mathcal{J}^*_{10}(An+B)\equiv 0 \pmod{2} \right\}}{X}=1
\end{align*}

\end{conj}
\subsection*{Acknowledgements} 
The author has carried out this work at the Indian Institute of Information Technology Sri City (IIITS). We would like to thank the referee for carefully reading our manuscript and giving such constructive comments, which substantially helped us to improve the quality of the article.
\bibliographystyle{plain}
\bibliography{Hauptmoduln}
\end{document}